\documentclass[12pt]{article}
\usepackage[DIV12]{typearea} 
\usepackage[utf8]{inputenc}
\usepackage{amsmath}
\usepackage{enumerate}
\usepackage{amssymb}
\usepackage{amsthm}
\usepackage{amsfonts}
\usepackage{textcomp}
\usepackage{stmaryrd}
\usepackage[all]{xy}

\DeclareMathOperator{\Ext}{Ext}

\DeclareMathOperator{\ch}{ch}
\DeclareMathOperator{\cH}{H}
\newcommand{\id}{\mathrm{id}}
\newcommand{\ta}[1]{#1^{[2]}}
\newcommand{\tb}[1]{#1^{[3]}}

\newcommand{\mc}[1]{\mathcal{#1}}

\newcommand{\tbmc}[1]{\tb{\mc{#1}}}

\DeclareMathOperator*{\NS}{NS}
\newcommand{\ZZ}{\mathbb{Z}}
\newcommand{\Z}{\mathbb{Z}}
\newcommand{\C}{\mathbb{C}}

\newcommand{\R}{\mathbb{R}}

\newcommand{\PP}{\mathbb{P}}
\newcommand{\CC}{\mathbb{C}}
\let\div\undefined
\DeclareMathOperator{\div}{div\!}

\DeclareMathOperator{\td}{td}
\newcommand{\Hilb}{\mathrm{Hilb}}
\makeatletter \renewenvironment{proof}[1][\proofname]
{\par\pushQED{\qed}\normalfont
\topsep6\p@\@plus6\p@\relax\trivlist \item[\hskip 1.5em \itshape#1\@addpunct{.}]\ignorespaces}{\popQED\endtrivlist\@endpefalse} \makeatother 
\newtheoremstyle{thm}
  {1.5\topsep}   % ABOVESPACE
  {1.5\topsep}   % BELOWSPACE
  {\itshape}  % BODYFONT
  {0pt}       % INDENT (empty value is the same as 0pt)
  {\bfseries} % HEADFONT
  {.}         % HEADPUNCT
  {5pt plus 1pt minus 1pt} % HEADSPACE
  {} 
\newtheoremstyle{defi}
  {1.5\topsep}   % ABOVESPACE
  {1.5\topsep}   % BELOWSPACE
  {\normalfont}  % BODYFONT
  {0pt}       % INDENT (empty value is the same as 0pt)
  {\bfseries} % HEADFONT
  {.}         % HEADPUNCT
  {5pt plus 1pt minus 1pt} % HEADSPACE
  {} 
\newtheoremstyle{rem}
  {1.5\topsep}   % ABOVESPACE
  {1.5\topsep}   % BELOWSPACE
  {\normalfont}  % BODYFONT
  {0pt}       % INDENT (empty value is the same as 0pt)
  {\bfseries} % HEADFONT
  {.}         % HEADPUNCT
  {5pt plus 1pt minus 1pt} % HEADSPACE
  {} 

\theoremstyle{thm}
\newtheorem{thm}{Theorem}[section]
\newtheorem*{thm*}{Theorem}
\newtheorem{prop}[thm]{Proposition}
\newtheorem{lem}[thm]{Lemma}

\newtheorem{cor}[thm]{Corollary}
\theoremstyle{defi}
\newtheorem{defi}[thm]{Definition}

\theoremstyle{rem}
\newtheorem{rem}[thm]{Remark}
\begin{document}

\title{Non-natural non-symplectic involutions on symplectic manifolds of $\ta{K3}$-type}
%moduli spaces of sheaves deformation equivalent to $K3^{[2]}$}
\author{Hisanori Ohashi\footnote{ohashi.hisanori@gmail.com}, Malte Wandel\footnote{wandel@math.uni-hannover.de}}

\maketitle
\begin{abstract}
We study non-symplectic involutions on irreducible symplectic manifolds of $K3^{[2]}$-type with $19$ parameters,
which is the second largest possible.
We classify the conjugacy classes of cohomological representations into four different types and show that there are  
at most five deformation types, two of which are given by natural involutions
and their flops. 
Next, we give a geometric realisation of one of the new types 
using moduli spaces of sheaves on $K3$ surfaces. The geometry of the manifold and the new involution is described in detail.
\vspace{2mm} \\
Keywords:\ moduli spaces, irreducible holomorphic symplectic manifolds, $K3$ surfaces, automorphisms.\\
MSC2010:\ 14D20, 14J28, 14J50, 14J60, 14F05.
\end{abstract}
\tableofcontents
\setcounter{section}{-1}
\section{Introduction}
During the last years automorphisms of irreducible holomorphic symplectic manifolds have been intensively studied. 
Inspired by the seminal work of Nikulin (\cite{Nik2}) on automorphisms of $K3$ surfaces, the theory has been developed by many authors such as Beauville (\cite{Beau,Beau2}), Boissi\'{e}re (\cite{Boi}),  Boissi\'{e}re$-$Nieper-Wisskirchen$-$Sarti (\cite{BNS,BNS2}), Camere (\cite{Cam}), Mongardi (\cite{Mon,Mon2}), O'Grady (\cite{O'G4}) and \linebreak Oguiso$-$Schr\"oer (\cite{OS}). The recent proof of the global Torelli theorem by Verbitsky (\cite{verbitsky}) will be the cornerstone in the subject.

Coming to non-symplectic involutions,
Beauville (\cite{Beau2}) was the first who studied them systematically. 
He showed that the fixed locus is a union of smooth Lagrangian submanifolds and in the case of fourfolds with $b_2=23$ he gave a classification of numerical invariants of the fixed locus, such as  
topological and holomorphic Euler characteristics, in terms of the trace 
of the action on the second cohomology.
He also gave a list of examples covering all possible values of the trace. 
In particular, it shows that the dimension of a family of involutions is at most $20$, which is attained by the famous 
double covers of EPW-sextics by O'Grady \cite{O'G}.

In this paper we study $19$-dimensional families of non-symplectic involutions in detail
using lattice theory, the moduli of marked manifolds and the theory of moduli spaces of sheaves on $K3$ surfaces. 
We note that the only known $19$-dimensional family so far has been the family of natural involutions. 
It turns out that there are several other deformation classes: We will distinguish them by the conjugacy classes 
of the cohomological action.
\begin{thm*}[Theorem \ref{4families}]
There exist four conjugacy classes of the action on the second cohomology of non-symplectic involutions with $19$ parameters 
on manifolds of $K3^{[2]}$-type. These classes can be distinguished by the invariant sublattice $H^2(X,\Z)^{\iota}$ as in the 
table in Theorem \ref{4families}. 

Moreover, in cases No.\ $1$, $2$ and $4$ any two involutions in the same class are deformation equivalent.
In case No.\ $3$, any involution can be deformed into one of two distinguished examples, namely the natural 
involution or its Mukai flop.
\end{thm*}
For the precise definition of deformation equivalence, see Definition \ref{deformation}. 
We note that the biregular involution on the flop is observed here for the first time.
The natural involution and its flop have the same invariants with respect to 
Proposition \ref{basic} and constitute the inseparable pair of points in the fibre of the period map. It would be interesting to ask the following question, which is a variant with automorphisms
of the fact that birational irreducible symplectic manifolds are deformation equivalent:\\
{\bf{Question:}} Are the natural involution and its Mukai flop deformation equivalent as 
automorphisms?

In this sense, our study introduces a new aspect in the classification theory of automorphisms
extending those for $K3$ surfaces.
In the proof of the theorem we make extensive use of the Torelli theorems. 

After the classification, as a new explicit example of a $19$-dimensional family, 
we give a realisation of involutions whose conjugacy class belongs to No.\ $1$ using moduli spaces 
of sheaves on the $19$-dimensional family of polarised $K3$ surfaces of degree two.
Note that by Remark \ref{impossibility} this idea exclusively applies to the case No.\ $1$.
We also show that the fixed locus consists of two connected components in this case. 
\begin{thm*}[Theorems \ref{thmfamily}, \ref{connected}]
There is a $19$-dimensional family of involutions whose conjugacy class belongs to No.\ $1$ in Theorem \ref{4families}.
It can be realised as a family of moduli spaces of sheaves on $K3$ surfaces of degree two. 
The fixed locus of these involutions consists of two smooth surfaces which are both a branched cover of $\PP^2$ of degree six and ten, respectively.
\end{thm*}

The paper is organised as follows. In the first section we collect basic results and fix the notation. 
Section 2 presents the results on classification and deformations of involutions with $19$ parameters. We show that there are four conjugacy classes of the action
on the second cohomology and, except for No. 3, two involutions in the same 
conjugacy class are deformation equivalent. We also give the descriptions of involutions
in No. 3.
In Section 3 some general considerations concerning induced automorphisms on moduli spaces of sheaves are given. The new example is constructed in Section 4. In Section 5 we show that the fixed locus consists of two smooth surfaces; the geometric description of these surfaces also distinguishes our example from the well known natural involutions on Hilbert schemes of two points.\\

\em Acknowledgements: \em 
The authors are grateful to K.\ Hulek and G.\ Mongardi for fruitful discussions and remarks. 
The first-named author wants to thank Leibniz university of Hannover and M.\ Sch\"{u}tt for giving him opportunities of 
visiting and studying. His research is partially supported by the JSPS Grant-in-Aid for Young Scientists (B) 23740010.
The second-named author wants to thank the DFG Research Training Group 1463 for supporting his stay in Japan, the IPMU (Tokyo) 
for excellent hospitality and M.\ Lehn and J.\ Kass for interesting discussions and helpful remarks.

\section{Preliminaries}\label{preliminaries}
In this first section we recall the most important notions needed in the sequel.

\begin{defi}
A compact complex K\"{a}hler manifold $X$ is called an \em irreducible holomorphic symplectic manifold \em (irreducible symplectic
manifold, for short) if $\pi_1(X)=\{1\}$ and $H^0(X,\Omega_X^2)\cong \CC \omega,$ where $\omega$ is a nowhere degenerate holomorphic symplectic two-form.
\end{defi}

\begin{thm}[Fujiki$-$Beauville$-$Bogomolov form]
For every irreducible symplectic manifold $X$ there exists a canonically defined non-degenerate pairing $(\,,\,)_X$ on $H^2(X,\ZZ)$ called \em Beauville$-$Bogomolov \em or sometimes \em Fujiki$-$Beauville$-$Bogomolov pairing.\em
\end{thm}

The most important examples of irreducible symplectic manifolds for this article are moduli spaces of sheaves on $K3$ surfaces. Let $S$ be a projective $K3$ surface. Mukai defined a lattice structure on
\[H^\ast (S,\ZZ)=H^0(S,\ZZ)\oplus H^2(S,\ZZ)\oplus H^4(S,\ZZ)\]
by setting 
\[(r_1,l_1,s_1).(r_2,l_2,s_2):=l_1.l_2-r_1s_2-r_2s_1.\]
We denote this lattice by $\widetilde{H}(S,\ZZ)$ and call it the \em Mukai lattice \em of $S$. Furthermore we can introduce a weight-two Hodge structure on $\widetilde{H}(S,\ZZ)$ by defining $\widetilde{H}^{1,1}:=H^0(S)\oplus H^{1,1}(S)\oplus H^4(S).$

Let $\mc{F}$ be a coherent sheaf on $S$. We can define its \em Mukai vector \em by 
\[v(\mc{F}):=\ch(\mc{F})\sqrt{\td_S}=(r_\mc{F},c_1(\mc{F}),r_\mc{F}+\ch_2(\mc{F})),\]
where $r_\mc{F},$ $c_1(\mc{F})$ and $\ch_2(\mc{F})$ denote the rank, the first Chern class and the $H^4$ component of the Chern character of $\mc{F},$ respectively. Now fix an element $v\in \widetilde{H}(S,\ZZ)$ and a polarisation $H\in\NS(S)$ and let $\mc{M}(v)$ denote the moduli space of $H$-semistable sheaves $\mc{F}$ satisfying $v(\mc{F})=v$. Denote by $\mc{M}^s(v)$ the open subset of $\mc{M}(v)$ parametrising only $H$-stable sheaves. If we want to emphasise the dependence on the polarisation $H$ we write $\mc{M}_H(v)$ for $\mc{M}(v)$.

\begin{rem}
Note that, using the same definition, the Mukai vector of a sheaf can be defined for any smooth variety $X$: For a sheaf $\mc{F}$ on $X$ we set $v(\mc{F}):=\ch(\mc{F})\sqrt{\td_X}\in H^\ast(X,\mathbb{Q}).$
\end{rem}

A special case of moduli spaces of sheaves are the Hilbert schemes of $n$ points on the surface, denoted by $\Hilb^n(S).$ The corresponding Mukai vector is $(1,0,-n+1)$ and for $n\geq 2$ we have an isometry of lattices
\[H^2(\Hilb^n(S),\ZZ) \cong H^2(S,\ZZ)\oplus \langle-2(n-1)\rangle\cong U^3\oplus E_8^2\oplus \langle -2(n-1)\rangle,\]
where on the left hand side the lattice structure is given by the Beauville$-$Fujiki form. 
For the notation of lattices, see below.
Irreducible symplectic manifolds which are deformation equivalent to $\Hilb^n(S)$ are called \em manifolds of $K3^{[n]}$-type.\em

By the work of Mukai, Huybrechts and O'Grady we have the following result:
\begin{thm}\label{thmdefeqhilb}
Let $(S,H)$ be a polarised $K3$-surface, $l\in\NS (S)$ a primitive class and $v=(r,l,s)\in\widetilde{H}(S,\ZZ)$ a Mukai vector such that $r\geq0$ and in the case $r=0,$ the class $l$ is effective. If furthermore $\mc{M}^s(v)=\mc{M}(v)$, then it is an irreducible symplectic manifold of dimension $v^2+2$ which is deformation equivalent to $\Hilb^n(S)$ with $n=\frac{1}{2}(v^2+2).$
\end{thm}
For a survey of the history and contributions of this result we refer to \cite[Sect.\ 6.2]{HL}.\\

Of course, in general it may be very difficult to ensure that there are no strictly semistable sheaves in $\mc{M}(v)$. We have the following general result (cf.\ Sect.\ 4.C in \cite{HL}):
\begin{prop}
Let $v=(r,l,s)$ be a Mukai vector such that $l\in\NS(S)$ is primitive. There is a locally finite set of hyperplanes --- so-called 'walls' --- inside the ample cone of $S$ such that $\mc{M}_H(v)=\mc{M}^s_H(v)$ if $H$ does not lie on any wall. Such an $H$ is called \em $v$-general.\em
\end{prop}

Furthermore, in \cite{O'G3} O'Grady studied the Hodge-structure on the second cohomology of $\mc{M}(v)$:
\begin{thm}\label{thmogrady}
Under the conditions of Theorem \ref{thmdefeqhilb} we have an isomorphism of integral Hodge structures and an isometry of lattices
\[H^2(\mc{M}(v),\ZZ)\cong v^{\perp}\subset\widetilde{H}(S,\ZZ).\] 
\end{thm}
\vspace{3mm}

For terminologies in lattice theory, our basic reference is \cite{Nik}. Here we recall the most important definitions in order to fix notation. 
A free abelian group $L$ of finite rank equipped with a (non-degenerate) 
symmetric bilinear form $(\,,\,)\colon L\times L\rightarrow \Z$ 
is called a {\em{lattice}}. We treat only {\em{even lattices}}, namely those which satisfy $(l^2)\in 2\Z$ for $l\in L$. 
The symbol $\oplus$ denotes the orthogonal direct sum of lattices.
We denote by $O(L)$ the group of isometries of $L$. 
Since $L$ is non-degenerate, the homomorphism $L\rightarrow L^*\colon l\mapsto (l,\cdot)$ is injective and yields the
residue group $A_L=L^*/L$, called the {\em{discriminant group}} of $L$. The finite group $A_L$ is naturally equipped with 
a quadratic form $q_L$ and a bilinear form $b_L$, see \cite{Nik} for the details. 
A lattice $L$ is {\em{unimodular}} if $|A_L|=1$. The unique even unimodular lattice of signature $(1,1)$ (resp. $(0,8)$) 
is denoted by $U$ (resp. $E_8$). The rank one lattice whose square of the generator is $n$ is 
denoted by $\langle n \rangle$. The latter will never be unimodular if $n$ is even. 
%The order of $A_L$ is equal to 
%the absolute value of the determinant of the Gram matrix of $L$. 
%On the finite group $A_L$, we have the induced quadratic form $q_L\colon A_L\rightarrow \Q /2\Z$ and the 
%bilinear form $b_L\colon A_L\times A_L \rightarrow \Q /\Z$. 
%The group $O(q_L)$ denotes all automorphisms of $A_L$ 
%which preserve the forms $q_L$ and $b_L$. 

The {\em{divisor}} or {\em{divisibility}} $\div_L (l)\in \mathbb{Z}_{\geq 0}$ of 
an element $l$ in a nondegenerate lattice $L$ is defined by 
\[(\div_L (l)) \mathbb{Z}=(l,L)=\{(l,l')\mid l'\in L\}\subset \Z.\quad \]
By definition $l^{*}:= l/\div_L(l)$ defines an element in the discriminant group $A_L$. 
The following criterion, usually referred to as {\em{Eichler's criterion,}} will be used repeatedly.
\begin{prop}\label{eichler}
Let $L$ be an even lattice containing $U\oplus U$ as a sublattice. Let $v,v'$ be primitive elements in $L$. 
If $(v^2)=(v'^2)$ and $v^*=(v')^*\in A_L,$ then there exists an isometry $\varphi\in O(L)$ which sends $v$ to 
$v'$. 
\end{prop}
\begin{proof}
\cite[Prop.\ 3.7.3]{scattone}.
\end{proof}

\vspace{3mm}

In this article we are interested in non-symplectic involutions on irreducible symplectic manifolds of $\ta{K3}$-type. 
Beauville (\cite{Beau2}) studied non-symplectic involutions in general, showing that the fixed locus 
consists of smooth Lagrangian submanifolds and 
gave the first systematic classification for manifolds of $\ta{K3}$-type as follows.
\begin{thm}[Beauville]\label{thmbeau}
Let $X$ be an irreducible symplectic fourfold with $b_2(X)=23$, $\sigma$ a non-symplectic involution on $X$ and 
$F$ the fixed surface. We denote by $e(X)$ the topological Euler characteristic. 
Let $t$ be the trace of $\sigma^*$ acting on $H^{1,1}(X)$. 
\begin{enumerate}
\item We have $K_F^2=t^2-1,\quad \chi (\mathcal{O}_F)=\frac{1}{8}(t^2+7),\quad e(F)=\frac{1}{2}(t^2+23)$. 
\item The local deformation space of $(X,\sigma)$ is smooth of dimension $\frac{1}{2}(21-t)$.
\item The integer $t$ takes any odd value between $-19$ and $21$. 
\end{enumerate}
\end{thm}
\begin{proof}
\cite[Thm.\ 2]{Beau2}.
\end{proof}

By this theorem, the maximal dimension of a family of non-symplectic involutions is $20$.
In fact, the double covers of EPW-sextics studied by O'Grady (\cite{O'G}) constitute a locally complete family of 
polarized irreducible symplectic manifolds of $\ta{K3}$-type, which attains this maximal dimension. 
The fixed surface is connected with the above characters for $t=-19$. 
We remark that Beauville's involutions (see \cite[Example 2.7]{Huy}) also 
belong to this family (\cite{O'G4}). 

In the second maximal case, namely for families of non-symplectic involutions of dimension $19$, we get
the following numerical characters.
\[K_F^2=288,\quad \chi (\mathcal{O}_F)=37,\quad e(F)=156.\]
The simplest example is discussed in \cite{Beau2}: Let $S\rightarrow \mathbb{P}^2$ be a double cover 
branched along a smooth curve $C\subset \mathbb{P}^2$ of degree $6$ and genus $10$. Let $\varphi$ be the covering
involution acting on $S$.
Then $\varphi$ induces a non-symplectic involution on $X=\Hilb^2(S),$ which moves in a family of dimension $19$. 
This involution is denoted by $\varphi^{[2]}$ and called the \em natural involution \em associated with $\varphi$ (\cite{Boi}).
The fixed surface $F$ is the disjoint union of $\Hilb^2 (C)$ and $\mathbb{P}^2=S/\varphi$.

In the following sections we study non-symplectic involutions with $19$-dimensional deformations in more detail.
(Equivalently we call them involutions with $19$ parameters.)
In particular we give a finer classification of them, study the deformation equivalence and realise another one of 
them using moduli spaces of sheaves on $K3$ surfaces.

\section{Deformation Equivalence of Involutions with $19$ Parameters}

We begin with clarifying the equivalence of automorphisms and its direct consequences.
\begin{defi}\label{deformation}
Let $X$ be a compact complex manifold and $\varphi \in \mathrm{Aut}(X)$.
{\em{A deformation of $\varphi$}} is a smooth family $f \colon \mathcal{X}\rightarrow \Delta$ of 
compact complex manifolds with $\mathcal{X}_0\simeq X$ 
together with an automorphism $\Phi\in \mathrm{Aut}(\mathcal{X})$ acting fibrewise such that we have $\Phi_0=\Phi|_{\mathcal{X}_0}\simeq \varphi$ at the base point $0\in \Delta$.
If we do not specify the base point, we simply call it a \em deformation \em (or a \em family\em ) \em of automorphisms. \em

Automorphisms $\varphi\in \mathrm{Aut}(X)$ and $\psi\in \mathrm{Aut}(Y)$ are {\em{connected by a deformation}} if 
there is a deformation of $\varphi$ 
which is at the same time a deformation of $\psi$ at another base point. They are said to be {\em{deformation equivalent}} if 
they can be connected via finitely many deformations of automorphisms.
\end{defi}
The following observations yield some invariants of deformation equivalence classes.
\begin{prop}\label{basic}
Let $\mathcal{X}\rightarrow \Delta$ and $\Phi\in \mathrm{Aut}(\mathcal{X})$ be a deformation of 
automorphisms, where we assume that $\Delta$ is connected.
\begin{enumerate}
\item The cohomology representations $\Phi_t^*$ and  $\Phi_s^*$ are conjugated by a parallel transport operator along 
a path connecting $t,s\in \Delta$. 
\item If $\Phi$ is of finite order, the collection of fixed loci $\{\mathrm{Fix}(\Phi_t)|t\in \Delta\}$ is a deformation of complex manifolds. 
In particular, the diffeomorphism type of $\mathrm{Fix}(\Phi_t)$ is constant.
\end{enumerate}
\end{prop}

The first assertion follows from the discreteness of the cohomology $H^i (\mathcal{X}_t,\Z)$, while the second can be shown by
linearising the action at fixed points relatively over the base.
By $1$., the conjugacy class of the action becomes an invariant for a family of automorphisms.
In this section, we refine the classification of non-symplectic involutions with $19$ parameters
for manifolds of $K3^{[2]}$-type taking into account this observation. \\

Let $\iota$ be a non-symplectic involution acting on an irreducible symplectic manifold $X$ of $K3^{[2]}$-type.
Then $X$ is automatically projective (see \cite{Beau2}) and the invariant sublattice 
\[H^2(X,\Z)^{\iota}:=\{x\in H^2(X,\Z)| \iota^* (x)=x \}\]
is hyperbolic, i.e.\ has signature $(1,n)$ for some $n$. If $\iota$ moves in a $19$-dimensional family, the 
invariant sublattice is of rank two. 
In this section we prove the following refinement to Theorem \ref{thmbeau}.
%\item There exists a unique conjugacy class of involutions in 
%$O(H^2(X, \Z ))$ for which the fixed sublattice $H^2(X,\Z )^{\iota}$ is
%positive of rank $1$. Geometrically, this corresponds to the non-symplectic involutions
%induced by the covering transformations of double EPW-sextics, see \cite{O'G}. 
%In the following, we exclude $20$-dimensional families and treat exclusively the case of $19$-dimensional families. 
%\item The Beauville's (birational) involutions on $\mathrm{Hilb}^2(S)$,  where $S$ is a quartic surface, 
%belong in fact to the above-mentioned covering transformations of double EPW-sextics. See \ref{O'G4}. 

\begin{thm}\label{4families}
There exist four conjugacy classes of the action on the second cohomology of non-symplectic involutions with $19$ parameters 
on manifolds of $K3^{[2]}$-type. These classes can be distinguished by the invariant sublattice $H^2(X,\Z)^{\iota}$ as follows:  
\begin{center}
\begin{tabular}{c|c|c}
No. & isom. class of $H^2(X,\Z)^{\iota}$ & property \\ \hline
$1$ & $U$ & \\
$2$ & $U(2)$ & \\
$3$ & $\langle 2 \rangle \oplus \langle -2 \rangle$ & $\div_{H^2(X,\Z)}(g)=2$ \\
$4$ & $\langle 2 \rangle \oplus \langle -2 \rangle$ & $\div_{H^2(X,\Z)}(g)=1$ \\
\end{tabular}
\end{center}
In No.\ $3$ and $4$, $g$ denotes the generator of $H^2(X,\Z)^{\iota}$ with $(g^2)=-2$.

Moreover, in cases No.\ $1$, $2$ and $4$ any two involutions in the same class are deformation equivalent.
In case No.\ $3$, any involution can be deformed into one of two distinguished examples, namely the natural 
involution (\cite{Boi}, see also Section \ref{preliminaries}) or its Mukai flop.
\end{thm}
\begin{rem}
\begin{enumerate}[(1)]
\item The family of natural involutions (see the sentences after Theorem \ref{thmbeau}) corresponds to No.\ $3$ as follows:
the involution $\varphi$ on a $K3$ surface $S$ has one-dimensional fixed sublattice isomorphic to
$\langle 2\rangle.$ The additional factor $\langle -2\rangle$ in the fixed lattice of the induced involution on the Hilbert square corresponds to half the class of the exceptional divisor of the Hilbert-Chow morphism. Alternatively we can use the description of $\Hilb^2(S)$ as a
moduli space of sheaves and apply Lemma \ref{lemogradyeq} below.
\item We emphasise that all deformation families here are families of biregular involutions, not only of birational involutions 
(see Definition \ref{deformation}).
In particular, by Proposition \ref{basic}, involutions in different conjugacy classes are never deformation equivalent. 
On the other hand, we do not know whether the family of natural involutions and its Mukai flop can be deformed to each other or not. See also Corollary \ref{cordefeq}.
\end{enumerate}
\end{rem}
\vspace{2mm}

Recall that the Beauville-Bogomolov lattice $H^2(X,\Z)$ of manifolds of $K3^{[2]}$-type is isomorphic to
\[\Lambda= U^3\oplus E_8^2\oplus \langle -2 \rangle.\]
We use $\Lambda$ as a reference lattice. Let us begin the proof of Theorem \ref{4families}.
\begin{prop}\label{involutions}
Let $\iota$ be an involution in $O(\Lambda )$ such that $\Lambda^{\iota}$ has signature $(1,1)$. Then $\iota$ is in one of four conjugacy classes, which are characterised by the table of Theorem \ref{4families}.
\end{prop}
\begin{proof}
First we prove that $\Lambda^{\iota}$ is $2$-elementary. Let $\tilde{\Lambda}$ be the unique 
unimodular overlattice
of $\Lambda\oplus \Z \lambda$, where $(\lambda^2)=2$. 
We can extend the action of $\iota$ to $\tilde{\Lambda}$ in such a way that $\iota$ acts on $\lambda$ by $-1$. We have $\Lambda^{\iota}=\tilde{\Lambda}^{\iota}$. 
Thus $\Lambda^{\iota}$ is the invariant sublattice of an involution acting on a unimodular lattice,
hence it is $2$-elementary. By the classification of $2$-elementary hyperbolic lattices, 
the isomorphism class of 
$\Lambda^{\iota}$ is one of $U, U(2)$ or $\langle 2 \rangle \oplus \langle -2 \rangle$.

No.\ $1$: If $\Lambda^{\iota}\simeq U$ is unimodular, it is an orthogonal summand in $\Lambda$ and 
the orthogonal complement is unique by \cite{Nik}, hence 
the embedding $\Lambda^{\iota}\subset \Lambda$ is unique. This shows that such $\iota$ constitute a single conjugacy class in $O(\Lambda)$.

No.\ $2$: Assume $\Lambda^{\iota}\simeq U(2)$. By looking at the discriminant forms, we see that 
$\Lambda^{\iota} \oplus \Z \lambda$ has no overlattices inside $\tilde{\Lambda}$. 
Therefore the orthogonal complement of $\Lambda^{\iota}$ inside $\Lambda$ is isomorphic to 
$N=U\oplus U(2)\oplus E_8^2\oplus \langle -2 \rangle$ by \cite{Nik}. Using the 
surjectivity of $O(N)\rightarrow O(q_N)$, it is easy to deduce that such $\iota$ constitute a single conjugacy class.

No.\ $3$ and $4$: If $\Lambda^{\iota}\simeq \langle 2 \rangle \oplus \langle -2 \rangle$ with generators
$h$ and $g$, we have two distinct orbits of $g$ under $O(\Lambda)$ characterised by $\div_{\Lambda} (g)=2$ or $1$. 
In the former case, we have that $g$ is on a unique orbit under $O(\Lambda )$ by Eichler's criterion (Proposition \ref{eichler}),
and $(\Z g)^{\perp}$ is isomorphic to the $K3$-lattice. Next again $h$ has a unique orbit by Eichler's criterion, hence 
we can see that No.\ 3 represents a single conjugacy class. 

Finally assume that $\Lambda^{\iota}$ satisfies the conditions of No.\ $4.$ 
By Eichler's criterion, $g$ is on a unique orbit and we get $(\Z g)^{\perp}\simeq 
\langle 2 \rangle \oplus U^2 \oplus E_8^2\oplus \langle -2 \rangle$. Here we need a computation.

{\bf{Claim.}} The divisor of $h$ in $(\Z g)^{\perp}$ is $1$. 

Assume to the contrary that $\div_{(\Z g)^{\perp}} (h)=2$.
Then for every $l\in \Lambda$ with $(h,l)$ odd, we have that the element $l':=\frac{1}{2}(2l+(g,l)g)$
is not an integral element in $(\Z g)^{\perp}$ since the pairing $(h,l')$ is odd. Thus necessarily
$(g,l)$ is odd. 

On the other hand, since $\div_{\Lambda}(g+h)=1$, we can find $m\in \Lambda$ such that $(g+h,m)=1$. 
By the above discussion, necessarily $(h,m)$ is even and $(g,m)$ is odd. 
Using the element $l$ as above, which exists since $\div_{\Lambda} (h)=1$, we find 
$(h,m+l)$ is odd and $(g,m+l)$ is even, a contradiction to the discussion above. 
Thus the claim is proved. 

The claim allows us to use Eichler's criterion again to see that $h$ has a unique orbit in $(\Z g)^{\perp}$,
and we get the uniqueness of the conjugacy class for No.\ 4.
\end{proof}
 
 We denote by $\Lambda^{\iota}_J$ the fixed sublattice in $\Lambda$ of an involution $\iota$ (taken as a representative
of the conjugacy class) of No.\ $J$ ($J=1,\dots,4)$. As a corollary to the Torelli theorems 
due to Verbitsky, Huybrechts and Markman, we have the existence of all cases as follows.
\begin{lem}\label{aruyo}
For each $J=1,\dots,4$ there exists (at least abstractly) an irreducible symplectic manifold of $K3^{[2]}$-type $X_J$ 
with an involutive automorphism $\iota_J$ whose conjugacy class in $\mathrm{Aut}(H^2(X_J,\Z))$ is 
in No. $J$. 
\end{lem}
\begin{proof}
By the surjectivity of period map, there exists a marked irreducible symplectic manifold
$(X_J, \eta)$ of $K3^{[2]}$-type such that 
$\eta$ induces an isomorphism $\NS(X_J)\stackrel{\sim}{\rightarrow} \Lambda^{\iota}_J\subset \Lambda$
(see also Lemma \ref{density}).
Via $\eta$, the involution $\iota$ acts on $H^2(X_J,\Z)$ and it is 
an orientation-preserving Hodge isometry 
of $H^2(X_J,\Z).$ Since $X_J$ is of $K3^{[2]}$-type, $\iota$ is a parallel transport operator (see \cite[Sect.\ 9]{markman}). 
The manifold $X_J$ is automatically projective by Huybrechts' criterion and $\iota$ acts trivially on $\NS(X_J).$ Hence it 
preserves a k\"{a}hler class and thus the action is realised by an automorphism of $X_J$ by the strong Torelli theorem.
\end{proof}

Let us proceed to the study deformations. We denote by $\mathfrak{M}$ the moduli space of 
marked manifolds of $K3^{[2]}$-type and by $\mathfrak{M}^0$ one of its connected components. Hereafter we fix 
a number $J\in\{1,2,3,4\}$.
Proposition \ref{involutions} shows that for every involution $\iota$ on some $X$ satisfying the condition of No.\ $J$, 
there exists an equivariant marking $\eta\colon H^2(X,\Z)\rightarrow \Lambda$ (where we denote both involutions on $H^2$ and 
$\Lambda$ by $\iota$)
so that the period $\eta (H^{2,0}(X))$
is in the following subset of the period domain $\Omega_{\Lambda}$,
\[\Omega_{(\Lambda^{\iota}_J)^{\perp}}=\{\C \omega\in \mathbb{P}((\Lambda^{\iota}_J)^{\perp} \otimes \C)\mid 
(\omega^2)=0, (\omega,\overline{\omega})>0 \}.\]
This set is a disjoint union of two symmetric domains of type IV. Let us fix one of them and denote it by $\Omega_J$. 
We can modify the marking $\eta$ by the next lemma, so that we have $\eta (H^{2,0}(X))\in \Omega_J$.
\begin{lem}\label{lemcomponents}
There is an isometry $\beta\in O(\Lambda)$ which commutes with $\iota$ and exchanges the two components 
of $\Omega_{(\Lambda^{\iota}_J)^{\perp}}$.
\end{lem}
\begin{proof}
Note that the two components correspond to two orientations of positive $2$-planes in $(\Lambda^{\iota}_J)^{\perp}\otimes \R$.
In each No., from the description in Proposition \ref{involutions},
we can see that the lattice $(\Lambda^{\iota}_J)^{\perp}$ is of the form
$U\oplus N'$ where $N'$ is some hyperbolic lattice. 
We can define an isometry $\alpha = (-1,1)\in O(U\oplus N')$ and it is easy to see that 
we can extend it to an isometry $\beta$ of $\Lambda$. 
\end{proof}
Furthermore, by replacing $\eta$ by $-\eta$ if necessary, we can assume that $(X,\eta)$ is in the connected component 
$\mathfrak{M}^0$. Thus, as a single marked manifold, $(X,\eta)$ can be located in the connected family $P_0^{-1} (\Omega_J)$,
where $P_0 \colon \mathfrak{M}^0\rightarrow \Omega_{\Lambda}$ is the period map.
In order to equip the family with a deformation of involutions,
we need the following lemmata.
\begin{lem}\label{smalldeformation}
For any involution $\iota$ on $X$ with the conjugacy class No.\ $J$, 
there exists a small neighborhood $N$ of $\eta (H^{2,0}(X_J))$ in $\Omega_J$ such that there exists a
smooth family of involutions over $N$ deforming $\iota$. Here $\eta$ is the marking 
constructed as above.
\end{lem}
\begin{proof}
The idea is from \cite[Section 4]{Mon}.
Let $f \colon \mathcal{X}\rightarrow \mathrm{Def}(X)$ be the Kuranishi family of $X$, where 
we can choose $\mathrm{Def}(X)$
to be an open ball. Note that $\mathrm{Def}(X)$ can be identified with an open neighborhood 
of $(X,\eta )$ in the moduli space $\mathfrak{M}$. 
By the universality, shrinking $\mathrm{Def}(X)$ if necessary,
the action of $\iota$ on $X$ extends to $\mathcal{X}$ and $\mathrm{Def}(X)$
so that $f$ is equivariant and the induced action on $\mathcal{X}_0\simeq X$ is $\iota$.
Then the restriction of $f$ to the fixed locus $\mathrm{Def}(X)^{\iota}$ is a deformation of involutions.
As in Theorem \ref{thmbeau}, $\mathrm{Def}(X)^{\iota}$ is smooth of dimension $19$. By the local Torelli theorem 
the period map sends $\mathrm{Def}(X)^{\iota}$ isomorphically to a submanifold of the same dimension in $\Omega_{\Lambda}$
locally around $\eta (H^{2,0}(X))$. Since this image must be contained in $\Omega_J$ which has the same dimension $19$,
they coincide. 
\end{proof}
\begin{lem}\label{density}
There exists a dense and connected subset $\Omega_J^0$ of $\Omega_J$ such that 
if the period of a marked irreducible symplectic manifold $(X, \eta)$
is in $\Omega_J^0,$ then $\eta$ induces an isomorphism $\NS(X)\stackrel{\sim}{\rightarrow} \Lambda^{\iota}_J$.
\end{lem}
\begin{proof}
For each $l\in (\Lambda^{\iota}_J)^{\perp}-\{0\}$ we define the hyperplane $H_l:=\{ \C\omega\in \Omega_J\mid(l,\omega)=0\}$ in $\Omega_J$ and 
set $\mathcal{H}=\cup H_l$. The subset $\Omega_J^0= \Omega_J - \mathcal{H}$ is dense by Baire's category theorem.
It is connected since $\mathcal{H}$ is a countable union of closed complex subspaces.
\end{proof}
By the proof of Lemma \ref{aruyo}, all $(X, \eta)$ with $\eta (H^{2,0}(X))\in \Omega_J^0$
carry an involution $\iota$ with the conjugacy class No. $J$. We call them {\em{generic}} involutions.

By Lemmata \ref{smalldeformation} and \ref{density} we can deform any given involution to a generic one. 
Again by the same lemmata, for any given two points in $\Omega_J^0,$ 
we can find a finite number of families of involutions such that their periods connect the given points. 
To see that two generic involutions from different families at an overlapping period
coincide, we have to study the fibre of the period map.
\begin{prop}\label{fibres}
Let $p\in \Omega_J^0$. Let $(X,\eta)$ be a marked manifold whose period is $p$. 
Recall that $P_0$ is the period map $\mathfrak{M}_0\rightarrow \Omega_{\Lambda}$ restricted to a 
connected component $\mathfrak{M}_0$.
Then the fibre $P_0^{-1}(p)$ is given as follows:

\begin{center}
\begin{tabular}{c|c|c}
No. & members of the fibre $P_0^{-1}(p)$ & $|P_0^{-1}(p)|$ \\ \hline
$1$ & $(X,\eta),$ $(X,s\eta)$ & $2$\\
$2$ & $(X,\eta)$ & $1$\\
$3$ & $(X,\eta),$ $(X,s\eta),$ $(X',\eta'),$ $(X',s\eta')$ & $4$\\
$4$ & $(X,\eta),$ $(X,s\eta)$ & $2$\\
\end{tabular}
\end{center} 

In No.\ $1,$ $3,$ $4,$ $s$ denotes the (unique) reflection in $(-2)$-vector in $\Lambda^{\iota}_J$. 
In particular, in No.\ $1,$ $2,$ $4$ the underlying manifolds in the fibre are the same
and only the marking is different.
\end{prop}
\begin{proof}
Elements in the fibre $P_0^{-1}(p)$ are in one-to-one correspondence with the K\"{a}hler-type chambers by (\cite[Sect.\ 5]{markman}).
For manifolds of $K3^{[2]}$-type, the exceptional chambers are cut out by $(-2)$-walls.
On the other hand, \cite{HT} shows that the ample cone contains one of the chambers which are cut out further 
by $(-10)$-walls with divisor $2$. 
It is easy to classify these walls for each classes.

In No.\ $1$ and $4,$ there is a unique $(-2)$-wall and no $(-10)$-walls, since the only $(-10)$-class in $\NS(X)$ has divisor $1$. 
Thus the positive cone has two chambers, exactly one of which is ample.
The reflection $s$ in the $(-2)$-wall gives a nontrivial monodromy element, and the result follows. 
In No.\ $2,$ there are no $(-2)$ nor $(-10)$ classes, hence there are no walls and we get $|P_0^{-1} (p)|=1$. 

In No.\ $3,$ we have both $(-2)$ and $(-10)$-walls which together separate the positive cone into four chambers. 
Note that by \cite[Lemma 3.4]{MM}, we have at least one Hilbert scheme 
in the fibre in this case. By the computation of the ample cone of Hilbert schemes by \cite[Lemma 13.3]{BM}, we see that in fact 
the fiber $P_0^{-1}(p)$ consists of four elements. They are as follows:
Since $s$ gives a nontrivial monodromy element, $(X,\eta)$ and $(X,s\eta)$
are two of them. Corresponding to the other chambers there is another bimeromorphic model $X'\sim X$, 
which, together with $\eta'$ and $s\eta'$, where $\eta'$ is the composite of 
$\eta$ with the isomorphism $H^2(X')\simeq H^2(X)$, gives the other points of the fibre.
\end{proof}
\begin{cor}\label{cordefeq}
Involutions in the same conjugacy class are deformation equivalent for No.\ $1,$ $2,$ $4.$ For No.\ $3,$ 
any involution can be deformed to one of two distinguished examples, namely the family of natural involutions or its Mukai flop.
\end{cor}
\begin{proof}
At an overlapping generic period we have two involutions $\iota_i$ on $X_i$ $(i=1,2)$ from different families. 
The above proposition shows that in fact $X_1\simeq X_2$ in No.\ $1,$ $2,$ $4.$ Since they are generic, the actions of $\iota_i$ on 
the second cohomology coincide via this isomorphism, i.e., they are given by 
the identity on the N\'{e}ron-Severi group and $-1$ on the orthogonal.
By the next lemma, the two involutions themselves coincide, and we can connect the families of involutions.
\begin{lem}\label{faithful}
Let $X$ be an irreducible symplectic manifold of $K3^{[2]}$-type.
Then the representation $\mathrm{Aut}(X)\rightarrow \mathrm{Aut} (H^2(X,\Z))$ 
is faithful.
\end{lem}
\begin{proof}
This should be known to experts; the proof is contained for example in \cite[Lem. 7.1.3]{Mon2}.
Alternatively we can imitate the proof of Lemma \ref{smalldeformation} and combine with \cite{Beau} and the density results
by \cite{MM} (Use the subset $B$ in the proof of Theorem 1.1 of \cite{MM}, where the fibres of the period map 
are all Hilbert schemes).
\end{proof}
We continue the proof of Corollary \ref{cordefeq} in No. 3. In this case, 
as Proposition \ref{fibres} shows, there exist two models of involutions in the same conjugacy class. 
The proof of deformation equivalence to one of these two is the same as above.
In the rest of proof we describe these involutions.

Let $\pi \colon S\rightarrow \mathbb{P}^2$ be a $K3$ surface doubly covering $\mathbb{P}^2$ branched along a smooth sextic. 
Let $\varphi$ be the covering involution acting on $S$.
The first involution, as mentioned several times, is the Hilbert scheme of two points of $S$ with the natural involution 
$\iota=\varphi^{[2]}$. It fixes $\mathrm{Hilb}^2(C)\subset \mathrm{Hilb}^2(S)$ and the surface
\[P:= \{\{ x,\varphi (x)\}\in \mathrm{Hilb}^2(S)\} \simeq S/\varphi= \mathbb{P}^2.\]
To obtain the other model, we perform the Mukai flop along $P$ and get $\psi \colon X\dashrightarrow X'$. 
By this operation, the projective plane $P$ is replaced 
by its dual $P^*=|\mathcal{O}_P(1)|$. It is easy to see that the involution $\iota$ acts regularly on 
$X'-P^*$ via $\psi$. What is important here, is 
that $\iota$ extends to a biregular involution of $X'$: let $Z\rightarrow X$ be the blowing up along $P$. We see that $\iota$ extends to $Z$ by fixing the exceptional divisor $E$ pointwise. Since $X'$ is given by the contraction of $E$ 
in another direction, $X'$ has a well-defined continuous action by $\iota$ fixing $P^*$. 
As a holomorphic automorphism, the (possibly) undefined set $P^*$ of $\iota$ 
on $X'$ is of codimension two, hence the action extends to a biregular action as stated. 
Let $\theta\colon H^2(X,\Z)\stackrel{\sim}{\rightarrow} H^2(X',\Z)$ be the isomorphism of integral Hodge structures 
induced by the flop. 
Since $\psi$ is not an isomorphism, $\theta$ does not preserve any k\"{a}hler class
by \cite{Fuj}. Since the period map sends both
$(X,\eta)$ and $(X',\eta\theta^{-1})$ to the same point, we see that these two models fill the fibres of Proposition \ref{fibres},
No. 3.
We note that $X'$ is obtained also as the relative compactified Jacobian 
of curves of genus two in the linear system
$|\pi^* \mathcal{O}_{\mathbb{P}^2}(1)|$ (cf.\ \cite{mukai}). The involution on the Jacobian of a smooth curve 
%in $|\pi^* \mathcal{O}_{\mathbb{P}^2}(1)|$ 
is induced from the hyperelliptic involution. Alternatively, the space $X'$ can be characterised as the moduli space of torsion sheaves with Mukai vector
$(0,\pi^* \mathcal{O}_{\mathbb{P}^2}(1),1)$. These interpretations show that $X'$ is in fact projective. 
%It is likely that this family of natural involutions and its Mukai flop are deformation equivalent, too. A partial reason 
%is that their fixed locus $\mathrm{Hilb}^2 (C)\cup \mathbb{P}^2$ is the same. But we have no proof of this statement.
\end{proof}

We end this section with the following remark:

\begin{prop}
Let $\iota$ be a generic involution on $X$. Then the full automorphism group $\mathrm{Aut}(X)$ and the birational automorphism group $\mathrm{Bir}(X)$ are both generated 
by $\iota$. 
\end{prop}
\begin{proof}
Let us pick up an element $\varphi\in \mathrm{Bir}(X)$. 
Since it acts isomorphically in codimension $1$, it acts on the second cohomology and 
gives an isometry (\cite{Huy}).
Since $X$ is projective, we can apply
\cite[Prop.\ 14.5]{ueno} to see that $\varphi$ acts on $H^{2,0}(X)$ by some primitive $N$-th root of unity. 
From \cite{Nik2} we deduce that $\phi (N)$ divides the rank of the orthogonal complement $T$ of $\NS(X)\subset H^2(X,\Z)$, where
$\phi$ is the Euler function. But here $T$ has rank $21$, which is an odd number, hence $N=1$ or $2$. 
In the latter case, \cite{Nik2} also shows that $\varphi|_T$ is the negation.
In particular, in any case, 
$\varphi|_T$ acts on the discriminant group $T^*/T$ trivially. See Proposition \ref{involutions} for the descriptions of $T$ in each case. 

On the other hand, the induced action on $\NS(X)$ preserves the 
fundamental exceptional chamber (the one which contains the ample cone) 
described in Proposition \ref{fibres} (\cite[Lemma 5.11]{markman}).
By simple computations, in No. $1$, $3$ and $4$ it follows that $\varphi$ is trivial on $\NS(X)$. 
Even in No. $2$, we see that the nontrivial isometry of $\NS(X)$ does not extend to $H^2$ by looking at the discriminant 
groups. Hence in this case also $\varphi$ must act trivially on $\NS(X)$. 

Thus the image of the representation
$\mathrm{Bir}(X)\rightarrow \mathrm{Aut} (H^2(X,\Z))$ consists only of two elements. 
With \cite{Fuj} and Lemma \ref{faithful}
we can conclude that $\mathrm{Bir}(X)=\mathrm{Aut}(X)=\{\mathrm{id}_X,\iota \}.$
\end{proof}

%%%%%%%%%%%%%%%%%%%%%%%%%%%%%%%%%%%%%%%%%%%%%%%%%%%%%%%%%%%%%%%%%%%%%%%%%%%%%%%%%%%%%%%%%%%%%%%%%%%%

\section{Induced Automorphisms on Moduli Spaces of Sheaves}
In this section some general aspects of automorphisms on moduli spaces of sheaves on $K3$ surfaces which are induced by automorphisms of the underlying surface are discussed.

The following proposition on induced automorphisms on moduli spaces is the central result of this section. Though we are only interested in the case of moduli spaces on $K3$ surfaces we state it in a more general setting:

\begin{prop}\label{propinducedauto}
Let $(X,H)$ be a polarised smooth projective variety, $\varphi$ an automorphism of $X$ preserving $H$ and $v\in \cH^\ast(X,\mathbb{Q})$ a Mukai vector. Then $\varphi$ induces a biregular isomorphism $\iota\colon\mc{M}^s(v)\rightarrow \mc{M}^s(\varphi^\star v).$ In particular, if $v$ is invariant under the induced action of $\varphi,$ we obtain a biregular automorphisms of $\mc{M}^s(v).$
\end{prop}

\begin{proof}
Pointwise the automorphism $\iota$ should map a stable sheaf $\mc{F}$ to the pullback $\varphi^\star\mc{F}.$ If $\mc{M}^s(v)$ is a fine moduli space, this assignment can be turned into a global morphism: Let $\mc{E}$ be a universal family on $X\times\mc{M}^s(v)$ We consider its pullback \linebreak$(\varphi\times\id_{\mc{M}^s(v)})^\star\mc{E}.$ This is a flat family of stable sheaves with Mukai vector $\varphi^\star v$ parametrised by $\mc{M}^s(v)$ and, by the universal property of $\mc{M}^s(\varphi^\star v),$ we get a classifying morphism $\iota\colon\mc{M}^s(v)\rightarrow\mc{M}^s(\varphi^\star v).$

In general a universal family only exists locally. It is not unique and hence we cannot glue local families. But, in fact, the associated local classifying morphisms are unique and by gluing we obtain a global automorphism of the moduli space. At this point I would like to thank J.\ Kass for proposing this strategy. Let us give some more details: For every point $[\mc{F}]$ in $\mc{M}^s(v)$ we can find an open neighbourhood $U$ and a universal family $\mc{E}_U$ on $X\times U$ representing the functor $\mc{M}^s_U(-)$ associating to a scheme $T$ the set of isomorphism classes of $T$-flat families $\mc{F}$ of stable sheaves such that for all closed points $t\in T$ the class of the restriction $\mc{F}_t$ is in $U.$ We pullback $\mc{E}_U$ along $\varphi\times\id_U$ and obtain a $U$-flat family of stable sheaves with Mukai vector $\varphi^\star v.$ Thus there is a classifying morphism $\iota_U\colon U\rightarrow \mc{M}^s(\varphi^\star v),$ which is obviously injective and maps the class of a sheaf $\mc{F}$ to the class of $\varphi^\star\mc{F}$ as above. Now we want to glue the morphisms $\iota_U.$  Thus consider two open subsets $U$ and $U'$ together with universal families $\mc{E}_U$ and $\mc{E}_{U'}$ and classifying morphisms $\iota_U$ and $\iota_{U'}$ and denote the intersection $U\cap U'$ by $V.$ By the universal property of the universal families we have $\mc{E}_U|_V\simeq\mc{E}_U'|_V\otimes p_V^\star\mc{L},$ where $\mc{L}$ is a line bundle on $V$ and $p_V\colon X\times V\rightarrow V$ denotes the second projection. On open subsets of $V$ where $\mc{L}$ is trivial, the two universal families are isomorphic and thus yield the same classifying morphism. Thus, in fact, the restrictions $\iota_U|_V$ and $\iota_{U'}|_V$ coincide and we can glue $\iota_U$ and $\iota_{U'}$ to obtain a morphism $\mc{M}^s(v)\rightarrow \mc{M}^s(\varphi^\star v).$\end{proof}

\begin{rem}\label{reminducedauto}
The proposition above can certainly be generalised to the relative setting: Let $\mc{X}\rightarrow T$ be a $T$-flat family of smooth projective varieties together with an automorphism $\varphi$ of $\mc{X}$ that acts fibrewise. Let $\mc{L}$ be a polarisation on $\mc{X}$ such that for every $t\in T$ the restriction $\mc{L}_t$ is preserved by the automorphism $\varphi_t$ on the fibre $\mc{X}_t.$ Furthermore let $v\in\cH^\ast(\mc{X},\ZZ)$ be a Mukai vector, such that for all $t\in T$ the pullback $v_t$ to the fibre $\mc{X}_t$ is $\varphi_t$-invariant. (It is enough to check this on one fibre if $T$ is connected.) Then $\varphi$ induces a biregular automorphism $\iota$ on the relative moduli space $\mc{M}_T^s(v)\rightarrow T$ such that the restriction $\iota_t$ coincides with the induced automorphism on $(\mc{M}_T^s(v))_t$ from Proposition \ref{propinducedauto} above.
\end{rem}

\begin{rem} If $\mc{M}^s(v)$ is compact and fine, the universal family $\mc{E}$ is simple. Thus if $\varphi$ is an involution, $\mc{E}$ admits a linearisation with respect to $(\varphi\times\iota)$.\end{rem}

Let us return to moduli spaces of sheaves on $K3$ surfaces.

\begin{lem}\label{lemogradyeq}
O'Grady's isomorphism $H^2(\mc{M}(v),\ZZ)\cong v^\perp$ of Theorem \ref{thmogrady} is equivariant with respect to $\varphi$ and $\iota$. Thus we can compute the invariant lattice as
\[H^2(\mc{M}(v),\ZZ)^\iota\cong(v^\perp)^\varphi.\]
\end{lem}
\begin{proof}
This is an easy consequence of the construction of the above isomorphism: Let $\mc{E}$ be a quasi-universal sheaf on $S\times\mc{M}(v)$ with multiplicity $\sigma$ and denote by $q\colon S\times\mc{M}(v)\rightarrow S$ and $p\colon S\times\mc{M}(v)\rightarrow \mc{M}(v)$ the natural projections. O'Grady defines a map $H^\ast(S,\mathbb{Z})\rightarrow H^2(\mc{M}(v),\mathbb{Z})$ by
\[\alpha\mapsto \frac{1}{\sigma}p_\star\big{[}q^\star\alpha\cdot\ch(\mc{E})\cdot q^\star\sqrt{\td_S}\big{]}_3, \hspace{10pt}\alpha\in H^\ast (S,\ZZ).\]
Here the $[-]_3$ indicates the projection onto $H^6(S\times\mc{M}(v),\mathbb{Z})$. Restricting this to $v^\perp$ yields the desired homomorphism, which is independent of the choice of $\mc{E}$. But by definition $\mc{E}$ is $(\varphi\times\iota)$-invariant. Thus $\ch(\mc{E})$ is invariant, too. Hence the lemma.
\end{proof}

The above observations lead to a quite general concept to construct and study automorphisms on moduli spaces of sheaves. This is applied in a very special situation in the next section. We want to point out that the method could be used in many other situations; for example also for moduli spaces of sheaves on abelian surfaces.

\section{The new Example}
In this section we construct a $19$-dimensional family of moduli spaces of sheaves on $K3$ surfaces using the methods of Section 3 and prove that it is a new example by means of the lattice theory developed in Section 2.

Let $\pi\colon S\rightarrow\mathbb{P}^2$ be a double cover branched along a smooth sextic curve $C$. This construction yields a $19$ dimensional family of $K3$ surfaces $S$ with involution $\varphi$ given by exchanging the covering sheets. Assume that $\NS(S)\cong\mathbb{Z}H$, where $H$ is the pullback of $\mc{O}_{\mathbb{P}^2}(1),$ thus $H^2=2$. The pullback of a general line $l\subset\PP^2$ is a smooth genus two curve in the linear system $|\mc{O}(H)|\cong |\mc{O}_{\PP^2}(1)|,$ the dual projective plane. Furthermore, from $\NS(X)\cong\mathbb{Z}H$ it follows that we only have three kinds of degenerations: If $l$ is tangent (bitangent) to $C,$ the pullback is an elliptic (rational) curve with one (two) ordinary double point(s) and if $l$ is tangent to $C$ in an inflection point, the pullback is an elliptic curve with a cusp. The sextic $C$ cannot have triple tangents since the pullback of such a line would split into two smooth rational curves in $S$ which would span a rank two lattice inside $\NS S.$ Note that, in particular, all curves in the linear system $|\mc{O}(H)|$ are reduced and irreducible. The locus of $S$ having Picard rank one is the complement of a countable union of closed subvarieties inside the moduli space of polarised $K3$ surfaces. And finally, it is well known that the involution $\varphi$ is non-symplectic. (If $\varphi$ preserved the symplectic form, the quotient would have to be symplectic as well.)

The involution $\varphi$ induces the natural involution $\ta{\varphi}$ on the corresponding Hilbert scheme of two points $\Hilb^2(S)$. By Lemma \ref{lemogradyeq} we see that the invariant lattice \linebreak $H^2(\Hilb^2(S),\ZZ)^\iota$ is spanned by $(0,H,0)$ and $(1,0,1)$. Thus it is isomorphic to $\langle 2 \rangle\oplus \langle -2 \rangle$. The second summand corresponds to half the class of the exceptional divisor in $\Hilb^2(S)$.\\

Now we come to the construction of the new example. Consider a length three subscheme $Z\subset S$ with ideal sheaf $\mc{I}_Z$.

\begin{lem}
We have 
\[h^1(\mc{I}_Z(H))=\begin{cases} 1 \text{ if $Z$ lies on a curve }D_Z\text{ in } |\mc{O}(H)|,\\ 0 \text{ otherwise}.  \end{cases}\]
\end{lem}

\begin{proof}
We have a short exact sequence
\[0\rightarrow \mc{I}_Z(H)\rightarrow \mc{O}_S(H)\rightarrow \mc{O}_Z\rightarrow 0.\]

The corresponding long exact sequence of cohomology starts with
\[0\rightarrow H^0(\mc{I}_Z(H))\rightarrow H^0(\mc{O}_S(H))\rightarrow H^0(\mc{O}_Z)\rightarrow H^1(\mc{I}_Z(H))\rightarrow 0.\]

Both terms in the middle are three dimensional and the map is just the evaluation map of the sections in the points of $Z$. Thus $h^1(\mc{I}_Z(H))\geq1$ if and only if $h^0(\mc{I}_Z(H))\neq0$, which exactly means that $Z$ is contained in a curve $D\in|\mc{O}(H)|$. But $Z$ cannot lie on two different curves $D\neq D'$ since $H^2=2<\mathrm{length}(Z)$ and $D$ and $D'$ cannot have common components because both are reduced and irreducible.
\end{proof}

Let us assume from now on that $Z$ is contained in a curve $D_Z$ in $|\mc{O}(H)|$. We therefore have a section $s\in H^0(\mc{I}_Z(H))$ and by the lemma a unique nontrivial extension
\begin{equation}\label{eqextension}
0\rightarrow \mc{O}_S\xrightarrow{\alpha} \mc{F} \rightarrow \mc{I}_Z(H)\rightarrow 0.
\end{equation}

\begin{lem}
Every such non-trivial extension $\mc{F}$ is stable.
\end{lem}

\begin{proof}
This follows from a more general argument, which can be found in \cite[Lemma 2.1]{Yosh}. For the convenience of the reader we shall repeat it in this special case. We only have to check rank one subsheaves. These are all of the form $\mc{L}=\mc{O}_S(aH)\otimes\mc{I}_{Z'}$ with $a\in\mathbb{Z}$ and $Z'\subset S$ a finite length subscheme. The slope of $\mc{F}$ is $1$, the slope of $\mc{L}$ is $2a$. Hence if $\mc{L}$ is destabilising, we must have $a\geq1$. If the induced map $\mc{L}\rightarrow \mc{I}_Z(H)$ is non-zero, we must have $a=1$. Thus $\mc{L}=\mc{I}_{Z'}(H)$. The resulting map $\Ext^1(\mc{I}_Z(H),\mc{O}_S)\rightarrow\Ext^1(\mc{I}_{Z'}(H),\mc{O}_S)$ maps the class of (\ref{eqextension}) to zero. (The destabilising map $\mc{I}_{Z'}(H)\rightarrow \mc{F}$ induces a splitting of the sequence.) But the kernel is $\Ext^1(\mc{I}_Z(H)/\mc{I}_{Z'}(H),\mc{O}_S)=0$. This is a contradiction to the fact that our extension was chosen to be non-trivial. Thus we get a map $\mc{L}\rightarrow\mc{O}_S$, but since $a\geq 1,$ this has to be zero, too.
\end{proof}

\begin{prop}\label{propsmooth}
Denote the Mukai vector of $\mc{F}$ by $v_0$. The moduli space $\mc{M}(v_0)$ is an irreducible symplectic manifold with an induced regular involution $\iota$. 
\end{prop}

\begin{proof}
Recall that the Picard rank of $S$ is one. Thus the ample cone of $S$ consists of one ray and there are no walls. (This follows since there are no elements in $\NS S$ of negative square (cf.\ Definition \cite[Def.\ 4.C.1]{HL})). Thus $\mc{M}(v_0)$ is an irreducible symplectic manifold. By Proposition \ref{propinducedauto} we have an induced regular involution.
\end{proof}

\begin{rem}
Note that in this special case we can deduce the regularity of the involution $\iota$ (as proven in Proposition \ref{propinducedauto}) directly: The invariant lattice of $\mc{M}(v_0)$ coincides with the N\'{e}ron$-$Severi group and thus any ample class is mapped to itself. Hence the birational involution $\iota$ is regular (cf.\ \cite[Cor.\ 3.3]{Fuj}).
\end{rem}

\begin{thm}\label{thmfamily}
There is a $19$-dimensional family of manifolds of $\ta{K3}$-type admitting a non-symplectic involution with invariant lattice isomorphic to $U$. Every member of this family is isomorphic to a moduli space of sheaves $\mc{M}(2,H,0)$ on a polarised $K3$ surface $(X,H)$ admitting a double cover to $\PP^2$. This family is different from the $19$-dimensional family of natural non-symplectic involutions on the Hilbert-schemes of two points.
\end{thm}

\begin{proof}
Let $\mc{K}_1$ denote the moduli space of polarised $K3$ surfaces of degree two. (Every such $K3$ surface is obtained as a smooth double sextic $\pi\colon S\rightarrow\PP^2$ and the polarisation $H$ is given by the pullback $\pi^\star\mc{O}_{\PP^2}(1).$) Note that $\mc{K}_1$ is an irreducible quasi-projective variety. Over $\mc{K}_1$ there does not exist a universal family. Following \cite[Lem. 2.7]{Sze}, there exists a finite cover $\mc{K}'\rightarrow\mc{K}_1$ with $\mc{K}'$ smooth together with a complete family $\psi\colon\mc{X}\rightarrow\mc{K}'$ of degree two $K3$ surfaces. The family $\mc{X}$ comes together with a polarisation $\mc{L}$ such that for all $t\in\mc{K}'$ the restriction $\mc{L}_t$ is just given by $H,$ the pullback of a line. As explained in \cite[Section 6.2]{HL}, we can construct a relative moduli space of sheaves $\rho\colon\mc{M}\rightarrow\mc{K}'$ on the fibres of $\psi$ such that for every point $t\in\mc{K}'$ the fibre $\mc{M}_t:=\rho^{-1}(t)$ is isomorphic to the moduli space $\mc{M}(2,c_1(\mc{L}_t),0)$ of sheaves on the surface $\mc{X}_t$. If $t$ corresponds to a surface of Picard rank one, we have seen that the moduli space is smooth. Moreover, the set of points $t,$ where $\rho$ is smooth, is open. Thus we find a Zariski dense open subset $\mc{K}^\circ\subseteq\mc{K}'$ such that the restricted family $\mc{M}^\circ$ constitues a $19$-dimensional family of irreducible symplectic manifolds.

The family $\mc{X}$ certainly carries a non-symplectic involution, which preserves the polarisation $\mc{L}.$ By Proposition \ref{propinducedauto} and Remark \ref{reminducedauto} we see that we have a biregular induced involution on the relative moduli space $\mc{M}^\circ\rightarrow \mc{K}^\circ$ acting fibrewise. If $t\in\mc{K}'$ is a point corresponding to a surface $\mc{X}_t$ of Picard rank one, then the involution on $\mc{M}_t$ is exactly the involution $\iota$ discussed above. From Lemma \ref{lemogradyeq} we immediately deduce that $\iota$ is non-symplectic. The Mukai vector of a sheaf $\mc{F}$ in an extension (\ref{eqextension}) is $v_0=v(\mc{F})=(2,H,0)$, so indeed its length is $2$ and the invariant lattice is generated by $(1,0,0)$ and $(0,H,1)$. Thus it is isomorphic to $U$.\end{proof}

\begin{rem} The construction (\ref{eqextension}) of $\mc{F}$ is a so-called Serre-construction, a correspondence between codimension two subschemes and rank two vector bundles as a generalisation to the correspondence between divisors and line bundles. The locus of degeneration of global sections of $\mc{F}$ define codimension two subschemes of $S$, the defining section $\alpha$ in (\ref{eqextension}) corresponds to the subscheme $Z$ itself.\end{rem}

\begin{rem}\label{impossibility}
We remark that the general involutions in the family No.\ $2$ and $4$ of Theorem \ref{4families}
cannot be realized as moduli spaces of sheaves whose involutions are induced 
from the surfaces. 

In fact, the unique $19$-dimensional family of $K3$ surfaces with involutions are double covers
$S\rightarrow \mathbb{P}^2$ branched along smooth sextics. If $S$ (or the sextic) is very general,
we have $\NS (S)= \Z H\simeq \langle 2 \rangle$, with $H$ the pullback of $\mathcal{O}(1)$. 
Suppose that for some Mukai vector $v$ on $S$, the irreducible symplectic manifold
$X=\mathcal{M}_H (v)$ had $H^2(X,\Z)^{\sigma}\simeq U(2)$. By O'Grady's isomorphism, we have 
$\NS (X)=H^2(X,\Z)^{\sigma}\simeq U(2)$ is given by the orthogonal complement 
of $v$ inside $H^0(S,\Z) \oplus \Z H \oplus H^4(S,\Z) \simeq U\oplus \langle 2 \rangle$. 
Thus we get the overlattice structure $\NS (X)\oplus \Z v\simeq U(2)\oplus \langle 2 \rangle
\subset U\oplus \langle 2 \rangle$. 
But this is not possible since the only isotropic subgroups in the discriminant group of $U(2)\oplus 
\langle 2 \rangle$ corresponds to $U(2)\subset U$, which contradicts the primitivity of $\NS (X)\simeq U(2)$. 

In the case of No.\ $4$ it suffices to show
that if we had on $S$ some Mukai vector $v$ such that the irreducible symplectic manifold $X$ had $\NS (X)\simeq 
\langle 2 \rangle \oplus \langle -2 \rangle$, then the length $-2$ generator $g\in \NS (X)$ has 
divisor $2$ in $\Lambda$. 
We have the overlattice structure $\NS (X)\oplus \Z v\subset U
\oplus \langle 2 \rangle.$ By looking at the discriminant forms, we see that $g$ 
satisfies $(g+v)/2\in H^*(S, \Z)$. It follows that $g/2\in (v^{\perp})^{*}=\Lambda^{*}$. 
This shows that $g$ has 
divisor $2$ inside $\Lambda=v^{\perp}$.
\end{rem}

\section{The Fixed Locus}
We continue with a few sheaf-theoretic considerations in order to understand the fixed locus $F$ of the involution $\iota$ on the moduli space $\mc{M}(v_0).$ %We denote by $\PP^{2\vee}$ the projective space of lines in $\PP^2$ which (by pullback) is isomorphic to $|\mc{O}(H)|.$

\begin{lem}\label{lemadddim}
We have $h^0(\mc{F})=2.$
\end{lem}

\begin{proof} This follows easily from (\ref{eqextension}) since $\cH^0(\mc{I}_Z(H))=\CC s$.\end{proof}

\begin{lem}\label{leminvnewexequiv}
Let $Z$ and $Z'$ be two length three subschemes lying on curves $D_Z$ and $D_{Z'},$ respectively. They define isomorphic sheaves $\mc{F}$ if and only if $D_Z=D_{Z'}$ and $\mc{O}_{D_Z}(-Z)\simeq\mc{O}_{D_{Z'}}(-Z').$
\end{lem}

\begin{proof} For the "if" direction note that $Z$ and $Z'$ yield exact sequences $0\rightarrow \mc{O}_S\xrightarrow{s}\mc{I}_Z(H)\rightarrow\mc{Q}\rightarrow 0$ and $0\rightarrow \mc{O}_S\xrightarrow{s'}\mc{I}_{Z'}(H)\rightarrow\mc{Q}'\rightarrow 0.$ The quotients $\mc{Q}$ and $\mc{Q'}$ are isomorphic to $\mc{O}_{D_Z}(H|_{D_Z}-Z)$ and $\mc{O}_{D_{Z'}}(H|_{D_{Z'}}-Z'),$ respectively. Thus by assumption $\mc{Q}\simeq\mc{Q}'.$ We get a diagram
\[\xymatrix{ & & 0\ar[d]&0\ar[d]&\\
&&\mc{O}_S\ar[d]^{\alpha'}\ar@{=}[r]&\mc{O}_S\ar[d]^s&\\
0\ar[r]&\mc{O}_S\ar[r]^\alpha\ar@{=}[d]&\mc{F}\ar[r]\ar[d]&\mc{I}_Z(H)\ar[r]\ar[d]&0\\
0\ar[r]&\mc{O}_S\ar[r]^{s'}&\mc{I}_{Z'}(H)\ar[r]\ar[d]&\mc{Q}\ar[r]\ar[d]&0\\
&&0&0&
}\]
with some sheaf $\mc{F}.$ Using the uniqueness of extensions of the form (\ref{eqextension}) we conclude. Conversely, starting with a sheaf $\mc{F}$ and exact sequences $0\rightarrow \mc{O}_S\xrightarrow{\alpha}\mc{F}\rightarrow \mc{I}_Z(H)\rightarrow 0$ and $0\rightarrow \mc{O}_S\xrightarrow{\alpha'}\mc{F}\rightarrow \mc{I}_{Z'}(H)\rightarrow 0,$ we immediately deduce from the diagram above that the quotient $\mc{Q}$ is $\mc{O}_{D_Z}(H|_{D_Z}-Z)\simeq\mc{O}_{D_{Z'}}(H|_{D_{Z'}}-Z')$.\end{proof}

Let $D\in|\mc{O}(H)|$ be a smooth curve and $Z\subset D$ a length three subscheme. We consider the degree three line bundle $\mc{O}_D(Z).$ Using Grothendieck$-$Riemann$-$Roch we can easily compute its Mukai vector:
\[v':=v(\mc{O}_D(Z))=(0,H,2).\]

\begin{lem}\label{leminvfixiso}
We have $\mc{M}(v_0)\cong \mc{M}(v').$ 
\end{lem}
\begin{proof}
Note that by the same argument as in the proof of Proposition \ref{propsmooth} the moduli space $\mc{M}(v')$ is an irreducible symplectic fourfold. We can make the isomorphism explicit (this was pointed out by M.\ Lehn): For a sheaf $\mc{F}\in\mc{M}(v_0)$ we consider the following short exact sequence:
\[0\rightarrow \mc{O}_S\otimes \cH^0(S,\mc{F})\xrightarrow{\mathrm{ev}} \mc{F}\rightarrow\mc{Q}\rightarrow0.\]
By Lemma \ref{lemadddim} the quotient $\mc{Q}$ is a torsion sheaf supported on a curve $D\in|\mc{O}(H)|,$ which coincides with $\mc{Q}\simeq\mc{O}_D(H|_D-Z)$ from the proof of Lemma \ref{leminvnewexequiv} above, for some choice of a length three subscheme $Z\subset D$ defined  by a section of $\mc{F}.$ Dualising $\mc{Q}$ and then tensoring with $\mc{O}_D(H|_D)$ gives the corresponding point in $\mc{M}(v')$. By Lemma \ref{leminvnewexequiv} this is an isomorphism.
\end{proof}

\begin{rem}\label{reminvfix}
The moduli space $\mc{M}(v')$ can be regarded as a relative compatified Jacobian: Denote by $\mc{U}\rightarrow|\mc{O}(H)|$ the universal family of curves in $|\mc{O}(H)|$ (cf.\ \cite[Sect.\ 3.1]{HL}). Inside $\tb{S}$ we have the relative Hilbert scheme $\tbmc{U}\rightarrow|\mc{O}(H)|$ of subschemes $Z\in\tb{S}$ lying on a curve $D_Z\in|\mc{O}(H)|.$ There is a relative compactified Abel map
\[f\colon\tbmc{U}\rightarrow \mc{M}(v')\]
over $|\mc{O}(H)|$ with generic fibre a projective line, which --- at least over points corresponding to smooth curves --- is given by the assignment
\[\tbmc{U}\ni Z\mapsto \mc{O}_{D_Z}(Z).\]
If $D$ is a smooth curve, the fibre over $D$ of the map
\[g\colon \mc{M}(v')\rightarrow|(\mc{O}(H)|,\quad \mc{O}_D(Z)\mapsto D\]
is precisely the Jacobian $\mc{J}^3D$ of degree three line bundles on $D.$

By Lemma \ref{lemadddim} each sheaf $\mc{F}$ defines a one dimensional family of length three subschemes in $X,$ which exactly corresponds to the fibre of $f$ over $\mc{F},$ where we regard $\mc{F}$ as a point in $\mc{M}(v')$ via the isomorphism $\mc{M}(v_0)\cong \mc{M}(v')$ of Lemma \ref{leminvfixiso}.

Certainly $\tbmc{U}$ is invariant under the induced involution $\tb{\varphi}$ on $\tb{S}$ and $f$ is equivariant with respect to $\tb{\varphi}$ and the action $\iota$ on $\mc{M}(v_0)\cong \mc{M}(v').$ Hence $g\colon\mc{M}(v')\rightarrow|\mc{O}(H)|$ is a $\iota$-invariant Lagrangian fibration. 
\end{rem}

\begin{thm}\label{connected}
The fixed locus of $(\mc{M}(v_0),\iota)$ consists of two smooth connected surfaces $F_1$ and $F_2$  which are both branched coverings of $\PP^2$ of degree six and ten, respectively.
\end{thm}

\begin{proof} 
We use the isomorphism $\mc{M}(v_0)\cong \mc{M}(v')$ of Lemma \ref{leminvfixiso}. By Remark \ref{reminvfix} we see that we have a $\iota$-invariant fibration $g\colon\mc{M}(v_0)\rightarrow|\mc{O}(H)|,$ where the action on the base is, of course, trivial. Therefore $\varphi$ acts on every fibre. A general point $D\in|\mc{O}(H)|$ corresponds to a smooth genus two curve which is a double cover of a line ramified at six points $p_1,\dots,p_6$ which are exactly the points in $D\cap C.$ Also they are the fixed points of the hyperelliptic involution $\iota_D$ on $D$. The fibre $g^{-1}(D)$ can be identified with the Jacobian $\mc{J}^3D$ and the involution on $\mc{J}^3D$ is given by pulling back divisors along $\iota_D.$ There are exactly $16$ fixed points in $\mc{J}^3D$ which are all of the form $p_i+p_j+p_k$ for some $i,j,k.$ We divide the set of fixed points into two sets. The first consists of the six classes of divisors of the form $3p_i$ for some $i.$ The second set consists of classes of divisors of the form $p_i+p_j+p_k$ with $i,j,k$ distinct. Now let $C^\star\subset |\mc{O}(H)|$ denote the locus of tangent lines to $C,$ i.e.\ the dual curve. As the curve $D$ moves in the open set $|\mc{O}(H)|\setminus C^\star$ the fixed points deform with it in the obvious way and it is impossible to deform a divisor of the first kind into one of the second kind. In this way we obtain two surfaces $\tilde{F}_1$ and $\tilde{F}_2$ which are unramified coverings of $|\mc{O}(H)|\setminus C^\star$ of degree six and ten.

What is left to do, is to show that the closures $F_1$ and $F_2$ of $\tilde{F}_1$ and $\tilde{F}_2,$ respectively, do not intersect and are both connected. (Here we want to thank Manfred Lehn for pointing out an error in an earlier version of this theorem and for indicating the beautiful proof of the revised statement.) For the first assertion we define a function on the set of degree three line bundles $\mc{L}$ on curves $D\in|\mc{O}(H)|\setminus C^\star$ as follows: For any such curve $D$ and any line bundle $\mc{L}$ on $D$ we consider the space of global sections $\cH^0(D,\mc{L}).$ The hyperelliptic involution $\iota_D$ acts by pullback on this vector space. We set $r(\mc{L}):=\dim \cH^0(D,\mc{L})^{\iota_D}.$ For a line bundle $\mc{L}_1$ corresponding to a divisor $3p_i$ we have $r(\mc{L}_1)=2$: Every divisor in the linear system $|\mc{L}_1|$ is of the form $p_i + p+\iota_D(p)$ for some $p\in D.$ This gives a two-dimensional space of sections which are all $\iota_D$-invariant. On the other hand, for a line bundle $\mc{L}_2$ associated with a divisor $p_i+p_j+p_k$ with $i,j,k$ distinct, the only fixed divisors in this linear system are $p_i+p_j+p_k$ and $p_l+p_m+p_n$ with $\{i,j,k,l,m,n\}=\{1,\dots,6\}.$ Thus $r(\mc{L}_2)=1.$ This analysis shows that the function $r$ takes different values on open parts of the surfaces $F_1$ and $F_2.$ Since the fixed locus is smooth, the two surfaces cannot intersect.

In order to show that $F_1$ and $F_2$ are connected, it is certainly enough to show that we can deform a divisor of the form $p_1+p_j+p_k$ on a smooth genus two curve $D$ into the divisor $p_2+p_j+p_k.$ We consider the situation that $D$ specialises to a nodal elliptic curve, where the points $p_1$ and $p_2$ come together. (This is the double cover picture of a general line in $\PP^2$ specialising to a tangent.) Denote the node by $q.$ The limits of the divisors $p_1$ and $p_2$ in the compactified Jacobian are both given by the divisor $q.$\end{proof}

\begin{rem}
In order to proof that the surfaces $F_1$ and $F_2$ do not intersect one can also proceed as follows: We want to show that we cannot deform a divisor of the form $2p_1$ to the divisor $p_1+p_2.$ Again, we can look at the limits of these divisors in the compactified Jacobian of the singular curve, where $p_1$ and $p_2$ come together to form a node $q.$ The support of both limits consists of the node $q,$ but the scheme structure is different. The limit of the divisor $p_1+p_2$ is a length two subscheme supported at $q$ consisting of the point $q$ together with a \em horizontal \em tangent vector. (Here we think of all curves being branched over a horizontal $\PP^1.$) On the other hand, the divisor $2p_1$ is a fibre of the two-to-one map to $\PP^1.$ Thus the limit point has to be a fibre, too. Indeed, the limit is a length two subscheme consisting of the node $q$ together with a \em vertical \em vector. This length two subscheme collapses when mapped to $\PP^1.$\end{rem}

\end{document}